\documentclass{ifacconf}

\usepackage{graphicx}      % include this line if your document contains figures
\usepackage{natbib}        % required for bibliography
\usepackage{amsmath}
\usepackage{amsfonts}
\usepackage{mathtools}
\usepackage{MnSymbol}
\usepackage{color}
\usepackage{float}
\ifdefined\beginlemma\else
\newtheorem{lemma}{Lemma}
\ifdefined\beginproposition\else
\newtheorem{proposition}{Proposition}
\ifdefined\begintheorem\else
\newtheorem{theorem}{Theorem}
\usepackage{lineno}
\makeatletter
\let\old@ssect\@ssect % Store how ifacconf defines \@ssect
\makeatother

\usepackage{natbib}
\usepackage{hyperref}

\makeatletter
\def\@ssect#1#2#3#4#5#6{%
  \NR@gettitle{#6}% Insert key \nameref title grab
  \old@ssect{#1}{#2}{#3}{#4}{#5}{#6}% Restore ifacconf's \@ssect
}
\makeatother
\graphicspath{{images/}}
\newcommand{\grad}{\normalfont{\text{grad}}}

\newcommand{\R}[1][\empty]{\mathbb{R}^{#1}}

\newenvironment{proof}{\paragraph{Proof:}}{\hfill$\square$}

\newcommand{\g}{\mathfrak{g}}

\newcommand{\Sg}{\mathcal{S}}
\newcommand{\Tg}{\mathcal{T}}

\newcommand{\Exp}{\normalfont{\text{Exp}}}
\newcommand{\Log}{\normalfont{\text{Log}}}
\newcommand{\ext}{\normalfont{\text{ext}}}
\newcommand{\Bi}{\normalfont{\text{Bi}}}

\newcommand{\N}{\mathbb{N}}

\newcommand{\tr}{\normalfont{\text{tr}}}
\newcommand{\SO}{\normalfont{\text{SO}}}

\makeatletter
\def\widebreve{\mathpalette\wide@breve}
\def\wide@breve#1#2{\sbox\z@{$#1#2$}%
     \mathop{\vbox{\m@th\ialign{##\crcr
\kern0.08em\brevefill#1{0.8\wd\z@}\crcr\noalign{\nointerlineskip}%
                    $\hss#1#2\hss$\crcr}}}\limits}
\def\brevefill#1#2{$\m@th\sbox\tw@{$#1($}%
  \hss\resizebox{#2}{\wd\tw@}{\rotatebox[origin=c]{90}{\upshape(}}\hss$}
\makeatletter

\DeclarePairedDelimiterX\set[1]\lbrace\rbrace{#1}

\ifdefined\begincorollary\else
\newtheorem{corollary}{Corollary}
\ifdefined\begindefinition\else
\newtheorem{definition}{Definition}
\ifdefined\beginremark\else
\newtheorem{remark}{Remark}
\begin{document}
\begin{frontmatter}

\title{Reduction of Sufficient Conditions in Variational Obstacle Avoidance Problems\thanksref{footnoteinfo}} 

\thanks[footnoteinfo]{The authors acknowledge financial support from Grant PID2022-137909NB-C21 funded by MCIN/AEI/ 10.13039/501100011033. The project that gave rise to these results received the support of a fellowship from ”la Caixa”
Foundation (ID 100010434). The fellowship code is LCF/BQ/DI19/11730028. Additionally, support has been given by the “Severo Ochoa Programme for Centres of Excellence”in R$\&$D (CEX2019-000904-S). }

\author[First]{Jacob R. Goodman} 
\author[Second]{Leonardo J. Colombo} 

\address[First]{J. Goodman is with Antonio de Nebrija University, Departamento de Informática, Escuela Politécnica Superior, C. de Sta. Cruz de Marcenado, 27, 28015, Madrid, Spain. email: jacob.goodman@nebrija.es}
   \address[Second]{L.Colombo is with Centre for Automation and Robotics (CSIC-UPM), Ctra. M300 Campo Real, Km 0,200, Arganda
del Rey - 28500 Madrid, Spain. email:  leonardo.colombo@csic.es}

\begin{abstract}                % Abstract of 50--100 words
This paper studies sufficient conditions in a variational obstacle avoidance problem on complete Riemannian manifolds. That is, we minimize an action functional, among a set of admissible curves, which depends on an artificial potential function used to avoid obstacles. We provide necessary and sufficient conditions under which the resulting critical points—the so-called modified Riemannian cubics—are local minimizers. We then study the theory of reduction by symmetries of sufficient conditions for optimality in variational obstacle avoidance problems on Lie groups endowed with a left-invariant metric. This amounts to left-translating the Bi-Jacobi fields described to the Lie algebra, and studying the corresponding bi-conjugate points. New conditions are provided in terms of the invertibility of a certain matrix.
\end{abstract}


\end{frontmatter}
%===============================================================================

\section{Introduction}
Path planning has become ubiquitous in fields such as robotics, industrial engineering, physics, biology, and related disciplines. Typically, we have a mechanical system governed by some physical laws or control schemes, and we wish for it to connect some set of knot points (interpolating given positions and velocities, and potentially higher order derivatives \cite{CLACC, CroSil:95}) while minimizing some quantity such as time or energy (e.g. battery consumption). For such problems, the use of variationally defined curves has a rich history due to the regularity and optimal nature of the solutions. In particular, the so-called \textit{Riemannian splines} \cite{noakes} are a particularly ubiquitous choice in interpolant, which themselves are composed of Riemannian polynomials—satisfying boundary conditions in positions, velocities, and potentially higher-order derivatives—that are glued together. In Euclidean spaces, Riemannian splines are just cubic splines. That is, the minimizers of the total squared acceleration.

Riemannian polynomials are smooth and optimal in the sense that they minimize the average square magnitude of some higher-order derivative along the curve. This quantity is often related to the magnitude of the controller in control engineering applications (which itself is related to energy consumption). Moreover, Riemannian polynomials carry a rich geometry with them, which has been studied extensively in the literature (see \cite{Giambo, marg, noakes} for a detailed account of Riemannian cubics and \cite{RiemannianPoly} for some results with higher-order Riemannian polynomials).

It is often the case that—in addition to interpolating points—there are obstacles or regions in space that need to be avoided. In this case, a typical strategy is to augment the action functional with an artificial potential term that grows large near the obstacles and small away from them (in that sense, the trajectories that minimize the action are expected to avoid the obstacles). This was done for instance in \cite{BlCaCoCDC}, \cite{BlCaCoIJC}, \cite{CoGo20}, \cite{colombo2023existence} where necessary conditions for extrema in obstacle avoidance problems on Riemannian manifolds were derived. In addition to applications to interpolation problems on manifolds and to energy-minimum problems on Lie groups and symmetric spaces endowed with a bi-invariant metric \cite{point}, and extended in \cite{mishal}, \cite{sh} and \cite{goodman2022collision} for the collision avoidance task and hybrid systems in \cite{goodman2021obstacle}. Reduction of necessary conditions for the obstacle avoidance problem were studied in \cite{goodman2022reduction} and sufficient conditions for the problem were studied in \cite{goodman2022sufficient}. In this paper, we build on the previous studies by first proving the converse result to Proposition 2 in \cite{goodman2022sufficient}, and then considering the problem of reduction by a Lie group of symmetries sufficient conditions for optimality in the variational obstacle problem on Lie groups endowed with a left-invariant metric, where a set of equivalent sufficient conditions are found in terms of the invertibility of a matrix whose elements are given by solutions to certain initial value problems. Finally, a brief study of the obstacle avoidance application is considered.

\section{Background on Riemannian manifolds}\label{Sec: background}

Let $(Q, \left< \cdot, \cdot\right>)$ be an $n$-dimensional \textit{Riemannian
manifold}, where $Q$ is an $n$-dimensional smooth manifold and $\left< \cdot, \cdot \right>$ is a positive-definite symmetric covariant 2-tensor field called the \textit{Riemannian metric}. That is, to each point $q\in Q$ we assign a positive-definite inner product $\left<\cdot, \cdot\right>_q:T_qQ\times T_qQ\to\mathbb{R}$, where $T_qQ$ is the \textit{tangent space} of $Q$ at $q$ and $\left<\cdot, \cdot\right>_q$ varies smoothly with respect to $q$. The length of a tangent vector is determined by its norm, defined by
$\|v_q\|=\left<v_q,v_q\right>^{1/2}$ with $v_q\in T_qQ$. For any $p \in Q$, the Riemannian metric induces an invertible map $\cdot^\flat: T_p Q \to T_p^\ast Q$, called the \textit{flat map}, defined by $X^\flat(Y) = \left<X, Y\right>$ for all $X, Y \in T_p Q$. The inverse map $\cdot^\sharp: T_p^\ast Q \to T_p Q$, called the \textit{sharp map}, is similarly defined implicitly by the relation $\left<\alpha^\sharp, Y\right> = \alpha(Y)$ for all $\alpha \in T_p^\ast Q$. Let $C^{\infty}(Q)$ and $\Gamma(TQ)$ denote the spaces of smooth scalar fields and smooth vector fields on $Q$, respectively. The sharp map provides a map from $C^{\infty}(Q) \to \Gamma(TQ)$ via $\grad f(p) = df_p^\sharp$ for all $p \in Q$, where $\grad f$ is called the \textit{gradient vector field} of $f \in C^{\infty}(Q)$. More generally, given a map $V: Q \times \cdots \times Q \to \R$ (with $m$ copies of $Q$), we may consider the gradient vector field of $V$ with respect to $i^{\text{th}}$ component as $\grad_i V(q_1, \dots, q_m) = \grad U(q_i)$, where $U(q) = V(q_1, \dots, q_{i-1}, q, q_{i+1}, \dots, q_m)$ for all $q, q_1, \dots, q_m \in Q$.

Vector fields are a special case of smooth sections of vector bundles. In particular, given a vector bundle $(E, Q, \pi)$ with total space $E$, base space $Q$, and projection $\pi: E \to Q$, where $E$ and $Q$ are smooth manifolds, a \textit{smooth section} is a smooth map $X: Q \to E$ such that $\pi \circ X = \text{id}_Q$, the identity function on $Q$. We similarly denote the space of smooth sections on $(E, Q, \pi)$ by $\Gamma(E)$. A \textit{connection} on $(E, Q, \pi)$ is a map $\nabla: \Gamma(TQ) \times \Gamma(E) \to \Gamma(TQ)$ which is $C^{\infty}(Q)$-linear in the first argument, $\R$-linear in the second argument, and satisfies the product rule $\nabla_X (fY) = X(f) Y + f \nabla_X Y$ for all $f \in C^{\infty}(Q), \ X \in \Gamma(TQ), \ Y \in \Gamma(E)$. The connection plays a role similar to that of the directional derivative in classical real analysis. The operator
$\nabla_{X}$ which assigns to every smooth section $Y$ the vector field $\nabla_{X}Y$ is called the \textit{covariant derivative} (of $Y$) \textit{with respect to $X$}.

Connections induces a number of important structures on $Q$, a particularly ubiquitous such structure is the \textit{curvature endomorphism}, which is a map $R: \Gamma(TQ) \times \Gamma(TQ) \times \Gamma(E) \to \Gamma(TQ)$ defined by $R(X,Y)Z := \nabla_{X}\nabla_{Y}Z-\nabla_{Y}\nabla_{X}Z-\nabla_{[X,Y]}Z$ for all $X, Y \in \Gamma(TQ), \ Z \in \Gamma(E)$.  The curvature endomorphism measures the extent to which covariant derivatives commute with one another. %We further define the \textit{curvature tensor} $\text{Rm}$ on $Q$ via $\text{Rm}(X, Y, Z, W) := \left<R(X, Y)Z, W\right>$.

We now specialize our attention to \textit{affine connections}, which are connections on $TQ$. Let $q: I \to Q$ be a smooth curve parameterized by $t \in I \subset \R$, and denote the set of smooth vector fields along $q$ by $\Gamma(q)$. Then for any affine connection $\nabla$ on $Q$, there exists a unique operator $D_t: \Gamma(q) \to \Gamma(q)$ (called the \textit{covariant derivative along $q$}) which agrees with the covariant derivative $\nabla_{\dot{q}}\tilde{W}$ for any extension $\tilde{W}$ of $W$ to $Q$. A vector field $X \in \Gamma(q)$ is said to be \textit{parallel along $q$} if $\displaystyle{D_t X\equiv 0}$. %For $k\in \N$, the $k$th-order covariant derivative of $W$ along $q$, denoted by $\displaystyle{D_t^k W}$, can then be inductively defined by $\displaystyle{D_t^k W = D_t \left(D_t^{k-1} W\right)}$. 

The covariant derivative allows to define a particularly important family of smooth curves on $Q$ called \textit{geodesics}, which are defined as the smooth curves $\gamma$ satisfying $D_t \dot{\gamma} = 0$. Moreover, geodesics induce a map $\mathrm{exp}_q:T_qQ\to Q$ called the \textit{exponential map} defined by $\mathrm{exp}_q(v) = \gamma(1)$, where $\gamma$ is the unique geodesic verifying $\gamma(0) = q$ and $\dot{\gamma}(0) = v$. In particular, $\mathrm{exp}_q$ is a diffeomorphism from some star-shaped neighborhood of $0 \in T_q Q$ to a convex open neighborhood $\mathcal{B}$ (called a \textit{goedesically convex neighborhood}) of $q \in Q$. It is well-known that the Riemannian metric induces a unique torsion-free and metric compatible connection called the \textit{Riemannian connection}, or the \textit{Levi-Civita connection}. Along the remainder of this paper, we will assume that $\nabla$ is the Riemannian connection. For additional information on connections and curvature, we refer the reader to \cite{Boothby}. When the covariant derivative $D_t$ corresponds to the Levi-Civita connection, geodesics can also be characterized as the critical points of the length functional $\displaystyle{L(\gamma) = \int_0^1 \|\dot{\gamma}\|dt}$ among all unit-speed \textit{piece-wise regular} curves $\gamma: [a, b] \to Q$ (that is, where there exists a subdivision of $[a, b]$ such that $\gamma$ is smooth and satisfies $\dot{\gamma} \ne 0$ on each subdivision). %Equivalently, we may characterize geodesics by the critical points of the \textit{energy functional} $\displaystyle{\mathcal{E} = \frac12 \int_a^b \|\dot{\gamma}\|^2 dt}$ among all $C^1$ piece-wise smooth curves $\gamma: [a, b] \to Q$ parameterized by arc-length. The length functional induces a metric $d: Q \times Q \to \R$ called the \textit{Riemannian distance} via $d(p, q) = \inf\{L(\gamma): \ \gamma \ \text{regular, } \gamma(a) = p, \ \gamma(b) = q\}$.

If we assume that $Q$ is \textit{complete} (that is, $(Q, d)$ is a complete metric space), then by the Hopf-Rinow theorem, any two points $x$ and $y$ in $Q$ can be connected by a (not necessarily unique) minimal-length geodesic $\gamma_{x,y}$. In this case, the Riemannian distance between $x$ and $y$ can be defined by $\displaystyle{d(x,y)=\int_{0}^{1}\Big{\|}\frac{d \gamma_{x,y}}{d s}(s)\Big{\|}\, ds}$. Moreover, if $y$ is contained in a geodesically convex neighborhood of $x$, we can write the Riemannian distance by means of the Riemannian exponential as $d(x,y)=\|\mbox{exp}_x^{-1}y\|.$ 

\subsection{Admissible Path Space}

The \textit{Lebesgue space} $L^p([0,1];\mathbb{R}^n)$, $p\in(1,+\infty)$ is the space of $\mathbb{R}^n$-valued functions on $[0,1]$ such that each of their components is $p$-integrable, that is, whose integral of the absolute value raised to the power of $p$ is finite. %\textcolor{red}{Equivalently, we can define it as the space of $\R^n$-valued functions on $[0, 1]$ whose Euclidean norm raised to the power of $p$ is finite.}
A sequence $(f_n)$ of functions in $L^p([0,1];\mathbb{R}^n)$ is said to be \textit{weakly convergent} to $f$ if for every $g\in L^r([0,1];\mathbb{R}^n)$, with  $\frac{1}{p}+\frac{1}{r}=1$, and every component $i$,  $\displaystyle{\lim_{n\to\infty}\int_{[0,1]}f_n^i g^i=\int_{[0,1]}f^i g^i}$. A function $g\colon[0,1]\to \mathbb{R}^n$ is said to be the \textit{weak derivative} of $f\colon[0,1]\to \mathbb{R}^n$ if for every component $i$ of $f$ and $g$, and for every compactly supported $\mathcal{C}^\infty$ real-valued function $\varphi$ on $[0,1]$, $\displaystyle{\int_{[0,1]}f^i\varphi'=-\int_{[0,1]}g^i\varphi}$.  The \textit{Sobolev space} $W^{k,p}([0,1];\mathbb{R}^n)$ is the space of functions $u\in L^p([0,1];\mathbb{R}^n)$ such that for every $\alpha\leq k$, the $\alpha^{th}$ weak derivative $\frac{d^\alpha u}{dt^{\alpha}}$ of $u$ exists and $\frac{d^\alpha u}{dt^{\alpha}}\in L^p([0,1];\mathbb{R}^n)$. In particular, $H^k([0,1];\mathbb{R}^n)$ denotes the Sobolev space $W^{k,2}([0,1];\mathbb{R}^n)$, and its norm may be expressed as $\displaystyle{\left|\left|f \right| \right| = \left(\int_{[0,1]} \sum_{p=0}^k \left| \left|\frac{d^k}{dt^k}f(t)\right|\right|_{\mathbb{R}^n}^2 dt \right)^{1/2}}$ for all $f \in H^k([0,1];\mathbb{R}^n)$, where $|| \cdot ||_{\mathbb{R}^n}$ denotes the Euclidean norm on $\mathbb{R}^n$. $(f_n)\subset W^{k,p}([0,1];\mathbb{R}^n)$ is said to be \textit{weakly convergent} to $f$ in $W^{k,p}([0,1];\mathbb{R}^n)$ if for every $\alpha\leq k$, $\displaystyle{\frac{d^\alpha f_n}{dt^{\alpha}}\rightharpoonup \frac{d^\alpha f}{dt^{\alpha}}}$ weakly in $L^p([0,1];\mathbb{R}^n)$.

%\textcolor{red}{Suppose that $|| \cdot ||$ is a norm on $\R^n$ which is bi-Lipschistz equivalent to the Euclidean norm $|| \cdot ||_{\R^n}$. That is, there exists $\alpha, \beta \in \R$ such that $\alpha ||x||_{\R^n} \le ||x|| \le \beta ||x||_{\R^n}$ for all $x \in \R^n$. Then, the space $L^p_{|| \cdot ||}([0, 1], \R^n)$ of $\R^n$-valued functions on $[0, 1]$ whose $|| \cdot ||$-norm to the power of $p$ is integrable is equivalent to $L^p([0, 1], \R^n)$ in the sense that any sequence $(x_n)$ which converges in one space (weakly or strongly) will converge in the other to the same limit.}

We denote by $H^2([0,1];Q)$ the set of all curves $q\colon[0,1]\to Q$ such that for every chart $(\mathcal{U},\varphi)$ of $Q$ and every closed subinterval $I\subset[0,1]$ such that $q(I)\subset\mathcal{U}$, the restriction of the composition $\varphi\circ q|_I$ is in $H^2([0,1];\mathbb{R}^m)$. Note that $H^2([0,1]; Q)$ is an infinite-dimensional Hilbert Manifold modeled on $H^2([0,1]; \mathbb{R}^m)$, and given $\xi = (q_0,v_0), \ \eta = (q_T, v_T) \in TQ$, the space $\Omega_{\xi, \eta}^T$ (denoted simply by $\Omega$ unless otherwise necessary) defined as the space of all curves $\gamma \in H^2([0,1]; Q)$ satisfying $\gamma(0) = q_0, \ \gamma(T) = q_T, \ \dot{\gamma}(0) = v_0, \ \dot{\gamma}(T) = v_T$ is a closed submanifold of $H^2([0,1]; Q)$ 

The tangent space $T_x \Omega$ consists of vector fields along $x$ of class $H^2$ which vanish at the endpoints together with their first covariant derivatives. $T_x \Omega$ has a natural Hilbert structure given by
$$\left<X, Y\right>_{T_x \Omega} = \left( \int_a^b \sum_{j=0}^2 g(D_t^j X, D_t^j Y)\right)^{1/2}.$$
Considering an orthonormal basis of parallel vector fields $\left\{\xi_i\right\}$ along $x$ and writing $X = X^i \xi_i$ and $Y = Y^i \xi_i$ for some coordinate functions $X^i, Y^i \in \mathring{H}^2([a,b]; \R) = \{f \in H^2([a,b]; \R) \ \vert \ f(a) = f'(a) = f(b) = f'(b) = 0\}$, we find that
$$\left<X, Y\right>_{T_x \Omega} = \left(\int_a^b \sum_{i=1}^n \left[ X^i Y^i + \dot{X}^i \dot{Y}^i + \ddot{X}^i \ddot{Y}^i \right]dt\right)^{1/2},$$
from which it is clear that $T_x \Omega$ can be identified with the Sobolev space $\mathring{H}^2([a,b], \mathbb{R}^n)$ (as discussed for instance in section 4.3 of \cite{Jost}).

%We may additionally consider a vector field $X \in T_x \Omega$ as the \textit{variational vector field} of some admissible variation of $x \in \Omega$.  corresponding to $\Gamma$ is defined by $\partial_s \Gamma(0, t)$. It can be seen that any $X \in T_\gamma \Omega$ is the variational vector field of some admissible variation, and conversely, the variational vector field of any admissible variation belongs to $T_\gamma \Omega$. Following the literature, we frequently denote an admissible variation of $\gamma$ by $\gamma_s$, and its corresponding variational vector field by $\delta \gamma$.

\subsection{Riemannian geometry on Lie Groups}\label{sec: background_Lie}

Let $G$ be a Lie group with Lie algebra $\g := T_{e} G$, where $e$ is the identity element of $G$. The left-translation map $L: G \times G \to G$ provides a group action of $G$ on itself under the relation $L_{g}h := gh$ for all $g, h \in G$. Given any inner-product $\left< \cdot, \cdot \right>_{\g}$ on $\g$, left-translation provides us with a Riemannian metric $\left< \cdot, \cdot \right>$ on $G$ via the relation:
\begin{align*}
    \left< X_g, Y_g \right> := \left< g^{-1} X_g, g^{-1} Y_g \right>_{\g},
\end{align*}
for all $g \in G, X_g, Y_g \in T_g G$. Such a Riemannian metric is called \textit{left-invariant}, and it follows immediately that there is a one-to-one correspondence between left-invariant Riemannian metrics on $G$ and inner products on $\g$, and that $L_g: G \to G$ is an isometry for all $g \in G$ by construction. Any Lie group equipped with a left-invariant metric is complete as a Riemannian manifold. In the remainder of the section, we assume that $G$ is equipped with a left-invariant Riemannian metric.

In the following $L_{g^{\ast}}$ stands for the push-forward of $L_g$, which is well-defined because $L_g: G \to G$ is a diffeomorphism for all $g \in G$. We call a vector field $X$ on $G$ \textit{left-invariant} if $L_{g\ast} X = X$ for all $g \in G$, and we denote the set of all left-invariant vector fields on $G$ by $\mathfrak{X}_L(G)$. It is well-known that the map $\phi: \g \to \mathfrak{X}_L(G)$ defined by $\phi(\xi)(g) = L_{g\ast} \xi$ for all $\xi \in \g, g \in G$ is an isomorphism between vector spaces. This isomorphism allows us to construct an operator $\nabla^{\g}: \g \times \g \to \g$ defined by:
\begin{align}
    \nabla^{\g}_{\xi} \eta := \nabla_{\phi(\xi)} \phi(\eta)(e),\label{g-connection}
\end{align}
for all $\xi, \eta \in \g$, where $\nabla$ is the Levi-Civita connection on $G$ corresponding to the left-invariant Riemannian metric $\left< \cdot, \cdot \right>$. Although $\nabla^{\g}$ is not a connection, we shall refer to it as the \textit{Riemannian $\g$-connection} corresponding to $\nabla$ because of the similar properties that it satisfies:
\begin{lemma}\label{lemma: covg_prop}
$\nabla^\g: \g \times \g \to \g$ is $\R$-bilinear, and for all $\xi, \eta, \sigma \in \g$, the following relations hold:
\begin{equation*}
 \hbox{(1) }\nabla_{\xi}^{\g} \eta - \nabla_{\eta}^{\g} \xi = \left[ \xi, \eta \right]_{\g},\, \hbox{(2) } \left< \nabla_{\sigma}^{\g} \xi, \eta \right> + \left<  \xi, \nabla_{\sigma}^{\g}\eta \right> = 0.
\end{equation*}
\end{lemma}

\begin{remark}\label{remark: covg_operator}
We may consider the Riemannian $\g$-connection as an operator $\nabla^\g: C^{\infty}([a, b], \g)\times C^{\infty}([a, b], \g) \to C^{\infty}([a, b], \g)$ in a natural way,  namely, if $\xi, \eta \in C^{\infty}([a, b], \g)$, we can write $(\nabla^\g_{\xi} \eta)(t) := \nabla^\g_{\xi(t)}\eta(t)$ for all $t \in [a, b]$. With this notation, Lemma \ref{lemma: covg_prop} works identically if we replace $\xi, \eta, \sigma \in \g$ with $\xi, \eta, \sigma \in C^{\infty}([a, b], \g)$.
\end{remark}

Given a basis $\{A_i\}$ of $\g$, we may write any vector field $X$ on $G$ as $X = X^i \phi(A_i)$, where $X^i: G \to \R$, where we have adopted the Einstein sum convention. If $X$ is a vector field along some smooth curve $g: [a, b] \to G$, then we may equivalently write $X = X^i g A_i$, where now $X^i: [a, b] \to \R$ and $g A_i =: L_g A_i$. We denote $\dot{X} = \dot{X}^i A_i$, which may be written in a coordinate-free fashion via $\dot{X}(t) = \frac{d}{dt}\left(L_{g(t)^{-1 \ast}}X(t) \right)$. We now wish to understand how the Levi-Civita connection $\nabla$ along a curve is related to the Riemannian $\g$-connection $\nabla^\g$. This relation is summarized in the following result \cite{goodman2022reduction}.

\begin{lemma}\label{lemma: cov-to-covg}
Consider a Lie group $G$ with Lie algebra $\g$ and left-invariant Levi-Civita connection $\nabla$. Let $g: [a,b] \to G$ be a smooth curve and $X$ a smooth vector field along $g$. Then the following relation holds for all $t \in [a, b]$:
\begin{align}
    D_t X(t) = g(t)\left(\dot{X}(t) + \nabla_{\xi}^\g \eta(t) \right).\label{eqs: Cov-to-covg}
\end{align}
\end{lemma}
\section{Sufficient conditions in the variational obstacle avoidance problem}

Consider a Riemannian manifold $Q$. The principle object of study along this chapter will be the functional $J: \Omega \to \R$ as:
\begin{equation}\label{J}
J[q]= \int_a^b \Big{(} \frac12 \left\|D_t \dot{q}(t)\right\|^2 + V(q(t))\Big{)}dt.
\end{equation}

where $V: Q \to \R$ is a smooth non-negative function called the \textit{artificial potential}. Of particular interest to us are the curves $q \in \Omega$ which minimize $J$. The critical points of $J$ can be found by considering the curves at which the differential of $J$ vanishes identically. That is, by finding the curves $q \in \Omega$ such that  $dJ[q]X = 0$ for all $X \in T_q \Omega$. Such a strategy (together with a boostrapping method for the purposes of regularity) were applied \cite{Goodmanthesis} to obtain the following result:

\begin{proposition}\label{Prop: Necessary conditions} A curve $q \in \Omega$ is a critical point of the functional $J$ if and only if it is smooth on $[a, b]$ and satisfies:
\begin{equation}\label{eq: necessary}
    D_t^3 \dot{q} + R\left(D_t \dot{q}, \dot{q} \right)\dot{q}=- \hbox{\grad} \, V(q(t)).
\end{equation}
\end{proposition}

We call smooth solutions to \eqref{eq: necessary} \textit{modified Riemannian cubics}. Now the problem remains to classify these critical points. In particular, we would like to understand when a modified Riemannian cubic minimizes $J$ (at least locally). For functions whose domain is a subset of some Euclidean space, demonstrating that a critical point is a local minimum amounts to applying the second-derivative test. As we will show in Theorem \ref{thm: sufficient conditions cubics index form}, the same principle applies for our problem—we need only replace the second derivative with the second variation (or differential) along a modified Riemannian cubic. For notational consistency, we introduce the following types of minimizers:

\begin{definition}\label{def: local minimizers}
A modified Riemannian cubic $q \in \Omega$ is a:
\begin{enumerate}
    \item[\hbox{(i)}] \textit{Global minimizer} of $J$ iff $J[q] \le J[\tilde{q}]$ for all $\tilde{q} \in \Omega$.
    \item[\hbox{(ii)}] $\Omega$-local minimizer of $J$ iff $J[q] \le J[\tilde{q}]$ for all $\tilde{q}$ in some $C^1$ neighborhood of $q$ (within $\Omega$).
    \item[\hbox{(iii)}] $Q$-local minimizer of $J$ iff for any $\tau \in [a, b]$, there exists an interval $[a^\ast, b^\ast] \subset [a, b]$ containing $\tau$ such that $q\vert_{[a^\ast,b^\ast]}$ is a global minimizer of $J$ on $\Omega_{\xi, \eta}^{[a^\ast, b^\ast]}$, where $\xi = (q(a^\ast), \dot{q}(a^\ast)), \ \eta = (q(b^\ast), \dot{q}(b^\ast))$.
\end{enumerate}
\end{definition}

It should be noted that we have slightly abused our notation in the definition of a $Q$-local minimizer. Technically, we are concerned with minimizing the integral $\displaystyle{\int_{a^\ast}^{b^\ast} \big{(}\frac12\left\|D_t\right\|^2 + V(q(t))\big{)}dt}$, which has different limits of integration than $J$ as defined in equation \eqref{J}. We will continue to refer to integrals of this form by $J$ throughout the paper—and in every case, the limits of integration will match that of the boundary conditions defined by the admissible set $\Omega_{\xi, \eta}^{[a, b]}$ on which $J$ is being discussed.

The $Q$-local minimizers were classified in their entirety in \cite{goodman2022sufficient}, where it was shown that a curve $q \in \Omega$ is a $Q$-local minimizers if and only if it is a modified Riemannian cubic. This is completely analogous to the fact that geodesics are the locally length-minimizing curves on a Riemannian manifold. In the next subsection, we will discuss the known results for $\Omega$-local minimizers. As we will see, these results are not nearly as complete: while sufficient conditions for optimality are obtained, they turn out to be quite difficult to work with in practice. The principal aim of Section \ref{section: reduction suff conds Lie groups} will then be to reduce these condition by symmetry so that they may be more readily studied in some special cases of interest.

\subsection{\texorpdfstring{$\Omega$}{Lg}-local minimizers}\label{subsection: omega-local}
We define the \textit{index form} $I: T_q \Omega \times T_q \Omega \to \R$ associated to the modified Riemannian cubic $q \in \Omega$ as $I(X, Y) = d^2 J[q](X, Y)$ for all $X, Y \in T_q \Omega$, where we've considered the second differential of $J$ as a bilinear map $d^2 J: T_q \Omega \times T_q \Omega \to \R$ via the identification $T_X (T_q \Omega) \cong T_q \Omega$. The "second-derivative test" from classical calculus then takes the following form with respect to the index form:

\begin{theorem}\label{thm: sufficient conditions cubics index form} 
Suppose that $q \in \Omega$ is a modified Riemannian cubic. If $I(X, X) > 0$ for all $X \in T_q \Omega \setminus \{0\}$, then $q$ is an $\Omega$-local minimizer of $J.$ 
\end{theorem}
\begin{proof}
First, suppose that $I(X, X) > 0$ for all $X \in T_q \Omega$. For some $\epsilon > 0$, consider an admissible variation $q_s$ of $q$ with variational vector field $\partial_s q_s \vert_{s=0} = X$, where $s \in (-\epsilon, \epsilon)$. Let $f(s) := J[q_s]$, and observe that $f'(0) = 0$ since $q$ is a modified Riemannian cubic. Moreover, $f''(0) = I(X, X) > 0$, so that by the second-derivative test, $f$ has a local minimum at $s=0$. It follows that $J[q] \le J[q_s]$ for all $s \in (-\epsilon, \epsilon).$ Since this holds for all admissible variations, it follows that $q$ is an $\Omega$-local minimizer. 
\end{proof}

In \cite{goodman2022sufficient}, the following explicit expression for the index form was obtained:

\begin{proposition}\label{prop: index form}
Let $q \in \Omega$ be a modified Riemannian cubic. Then the index form along $q$ is given by
\begin{equation}\label{eq: index form}
    I(X, Y) = \int_a^b \left[\left<D_t^2 X, D_t^2 Y \right> + \left<F(X, \dot{q}) + \nabla_X \grad V, Y \right>\right]dt,
\end{equation}
for all $X, Y \in T_q \Omega$, where 
\small{\begin{align}
    F(X, Y) &= (\nabla^2_Y R)(X, Y)Y + (\nabla_X R)(\nabla_Y Y, Y)Y + R(R(X, Y)Y, Y)Y \nonumber\\&+ R(X, \nabla^2_Y Y)Y + 4R(\nabla_Y X, Y)\nabla_Y Y\label{eq: F}\\
    &+ 2\left[(\nabla_Y R)(\nabla_Y X, Y)Y + (\nabla_Y R)(X, \nabla_Y Y)Y + R(\nabla^2_Y X, Y)Y \right]\nonumber \\
    &+ 3\left[ (\nabla_Y R)(X, Y)\nabla_Y Y + R(X, Y) \nabla_Y^2 Y + R(X, \nabla_Y Y)\nabla_Y Y \right]  \nonumber.
\end{align}}
\end{proposition}

Verifying the conditions of Theorem \ref{thm: sufficient conditions cubics index form} using \eqref{eq: index form} will not be possible in general. For that reason, we turn our attention to the kernel elements of the index form:

\begin{proposition}\label{prop: kernel index form}
A vector field $X \in T_q \Omega$ belongs to the kernel of $I$ if and only if $X$ is smooth and satisfies 
\begin{equation}\label{eq: bi-Jacobi field}
    D_t^4 X + F(X, \dot{q}) + \nabla_X \grad V(q) = 0
\end{equation} for all $t \in [a, b]$. 
\end{proposition}
\begin{definition}
A vector field $X$ along a modified cubic $q$ satisfying $\frac{D^4}{dt^4} X + F(X, \dot{q}) + \nabla_X \grad V(q) \equiv 0$ on $[0, T]$ is called a modified bi-Jacobi field.
\end{definition}

Observe that in the case where $V \equiv 0$, the definition of a modified bi-Jacobi Field coincides with that of a bi-Jacobi field, as defined in \cite{Sufficient2001}. Moreover, the equation describing the modified bi-Jacobi fields is linear in $X$, so that (since $V$ is smooth) the modified bi-Jacobi fields are smooth and the existence and uniqueness of solutions on $[0, T]$ given initial values $X(0), \ \frac{D}{dt}X(0), \ \frac{D^2}{dt^2}X(0), \ \frac{D^3}{dt^3}X(0)$ follows immediately (say, by moving to coordinate charts). In particular, the set of modified bi-Jacobi fields along a modified cubic polynomial $q$ forms a $4n$-dimensional vector space. Modified bi-Jacobi fields are particularly useful when paired with the concept of \textit{biconjugate points}:

\begin{definition}
Two points $t = t_1, t_2 \in [0, T]$ are said to be biconjugate along a modified cubic $q$ if there exists a non-zero modified bi-Jacobi field $X$ such that
\begin{align*}
    X(t_1) = X(t_2) = 0, \quad \text{ and }\qquad \frac{D}{dt}X(t_1) = \frac{D}{dt}X(t_2) = 0.
\end{align*}
\end{definition}

\noindent Analogous to the case of geodesics and conjugate points (\cite{Jost}, Theorem 4.3.1), or Riemannian cubic polynomials and biconjugate points (\cite{Sufficient2001}, Theorem 7.2), we have shown in \cite{goodman2022sufficient} that modified cubic polynomials do not minimize past their biconjugate points:

\begin{theorem}\label{thm: not_min}
Suppose that $q \in \Omega$ is a modified Riemannian cubic and that $a \le t_1 < t_2 \le b$ are biconjugate. Then $q$ is not an $\Omega$-local minimizer of $J$.
\end{theorem}

Here we show that the converse is also true. 

That is, we would like to show that if there are no biconjugate points along a modified Riemannian cubic $q$, then $q$ is an $\Omega$-local minimizer. Before proceeding to the result, we introduce a symmetry of modified bi-Jacobi fields that will help to simplify calculations. Let $\alpha_q(X, Y) := \langle D_t^3 X, Y \rangle + 2\langle R(D_t X, \dot{q})\dot{q}, Y\rangle + \langle D_t X, D_t^2 Y\rangle + 2\langle R(X, \dot{q})D_t \dot{q}, Y \rangle$.

\begin{lemma}\label{Bijacobi symmetry}
If $X$ and $Y$ are bi-Jacobi fields along $q$ such that $X(0) = Y(0) = 0$ and $D_t X(0) = D_t Y(0) = 0$, then 
\begin{equation}\label{eq: antisymmetric index term symmetry}
    \frac12\big{(}\alpha_q(X, Y) - \alpha_q(Y, X)\big{)} = P_{-}(X, Y)
\end{equation}
\end{lemma}

\begin{proof}
Observe that 
\small{\begin{align*}
    \langle D_t X + F(X, \dot{q}) + \nabla_X \grad V, Y \rangle - \langle D_t Y + F(Y, \dot{q}) + \nabla_Y \grad V, X \rangle = 0
\end{align*}}
\normalsize since $X$ and $Y$ are bi-Jacobi fields along $q$. We may separate this as
\small{\begin{equation}\label{eq: potential_symmetry}
    \langle D_t X + F(X, \dot{q}), Y \rangle - \langle D_t Y + F(Y, \dot{q}), X \rangle = \langle \nabla_Y \grad V, X \rangle - \langle \nabla_X \grad V, Y \rangle.
\end{equation}}
\normalsize In \cite{Sufficient2001}, it was shown that
\small{\begin{align*}
    \langle D_t X + F(X, \dot{q}), Y \rangle - \langle D_t Y + F(Y, \dot{q}), X \rangle &= D_t\left[\alpha_q(X, Y) - \alpha_q(Y, X) \right].
\end{align*}}
\normalsize Using \eqref{eq: potential_symmetry} and integrating over $t$ from $t=a$ to $t=b$, we obtain the desired result. 
\end{proof}

\begin{theorem}\label{thm: suff_conds_biconjugate}
Suppose that $q \in \Omega$ is a modified Riemannian cubic, and $t =a$ and $t=t_0$ are not biconjugate along $q$ for each $t_0 \in(a, b].$ Then $q$ is an $\Omega$-local minimizer.
\end{theorem}

\begin{proof}
Suppose that $X \in H^2_g(q)$ and satisfies $X(a) = D_t X(a) = 0.$ We will show that, for the bi-Jacobi field $J$ along $q$ satisfying $J(a) = D_t J(a) = 0, \ J(b) = X(b), \ D_t J(b) = D_t X(b)$—which exists and is uniquely defined since there are no biconjugate points along $q$ (see \cite{Sufficient2001}), we have $0 = I(J, J) \le I(X, X)$ with equality if and only if $J = X.$ The result then follows from the fact that, if there are no biconjugate points, there is no non-zero bi-Jacobi field in $T_q \Omega$—and hence the equality $J = X$ cannot hold for any non-zero $X \in T_q \Omega.$

Following the strategy implemented in \cite{Sufficient2001}, let $\{v_i\}_{i = 1}^n$ be a basis for $T_{q_b} Q$ and consider the bi-Jacobi fields $\{J_i\}_{i=1}^{2n}$ defined by
\begin{align*}
    &J_i(a) = 0,  &&D_t J_i(a) = 0, \\
    &J_i(b) = v_i,  &&D_t J_i(b) = 0, \qquad \text{for} \ i=1, \dots, n\\
    &J_i(a) = 0,  &&D_t J_i(a) = 0, \\
    &J_i(b) = 0,  &&D_t J_i(b) = v_{i-n}, \qquad \text{for} \ i=n+1, \dots, 2n.
\end{align*}
Since there are no biconjugate points along $q$, these $2n$ bi-Jacobi fields are uniquely defined and linearly independent, and hence form a basis for the vector space $J_q(a)$ of bi-Jacobi fields along $q$ which vanish at $t = a$ along with their first covariant derivatives. Thus, $J = c^i J_i$ for some real numbers $c^i, \ i = 1, \dots 2n.$ Moreover, it is clear that $(J_i(t_0), D_t J_i(t_0))$ forms a basis for $T_{q(t_0)} Q \times T_{q(t_0)} Q$ for each $t_0 \in (a, b]$ since $t=a$ and $t=t_0$ are not bi-conjugate. Hence, we may write $(X(t), D_t X(t)) = \sum_{i=1}^{2n} f^i(t) (J_i(t), D_t J_i(t))$ for all $t \in [a, b]$ where $f^i \in H^2([a, b], \R)$ is such that $f^i(b) = c^i$ for all $i = 1, \dots, 2n$, and $\sum_{i=1}^{2n} f'_i(t) J_i(t) = 0$ for all $t \in [a, b].$ Observe that
\begin{align}
    D_t^2 X &= \dot{f}^i D_t J_i + f^i D_t^2 J_i, \\
    F(X, \dot{q}) &= f^i F(J_i, V) + 2 \dot{f}^i R(D_t J_i, \dot{q})\dot{q}.
\end{align}
Hence,
\small{\begin{align*}
    I(X, X) &= \int_a^b \Big{[} \| \dot{f}^i D_t J_i \|^2 + \langle \dot{f}^i D_t J_i, f^j D_t^2 J_j \rangle + \| f^i D_t^2 J_i\|^2 \\ &+ \langle f^i J_i, f^j F(J_j, \dot{q}) + f^j \nabla_{J_j} \grad V \rangle + 2\langle f^i J_i, \dot{f}^j R(D_t J_j, \dot{q})\dot{q} \rangle \Big{]}dt.
\end{align*}}
\normalsize Observe that 
\small{\begin{align*}
    \| f^i D_t^2 J_i \|^2 &= D_t \Big{[} \langle f^i D_t J_i, f^j D_t^2 J_j \rangle - \langle f^i J_i, f^j D_t^3 J_j \rangle \Big{]} \\
    &+ \langle f^i J_i, \dot{f}^j D_t^3 J_j \rangle + \langle f^i J_i, f^j D_t^4 J_j \rangle + \langle \dot{f}^i J_i, f^j D_t^3 J_j \rangle \\
    &- \langle \dot{f}^i D_t J_i, f^j D_t^2 J_j\rangle - \langle f^i D_t J_i, \dot{f}^j D_t^2 J_j \rangle
\end{align*}}
\normalsize Substituting this identity into $I(X,X)$ and making use of the fact that $J_i$ is a bi-Jacobi field for each $i = 1, \dots, 2n$, we obtain
\small{\begin{align*}
    I(X, X) &= \left\langle c^i D_t J_i(T), c^j D_t^2 J_j(T) \right\rangle - \left\langle c^i J_i(T), c^j D_t^3 J_j(T) \right\rangle \\
    &+ \int_a^b \Big{(}\|\dot{f}^i D_t J_i \|^2 + \dot{f}^i f^j \Big{[} \langle D_t^3 J_i, J_j \rangle - \langle D_t^3 J_j, J_i \rangle \\
    &+ \langle D_t J_i, D_t^2 J_j \rangle - \langle D_t J_j, D_t^2 J_i \rangle + 2 \langle R(D_t J_i, \dot{q})\dot{q}, J_j \rangle\Big{]}\Big{)}dt.
\end{align*}}
\normalsize The first line in the expansion is simply $I(J, J)$, which can be seen by integrating $I(J, J)$ twice by parts. Moreover, since we have $\dot{f}^i J_i = 0$, it follows that
\begin{align*}
    &\dot{f}^i f^j\langle R(D_t J_j, \dot{q})\dot{q}, \dot{f}^i f^j J_i \rangle = 0, \\
    &\dot{f}^i f^j\langle R(J_i, \dot{q})D_t \dot{q}, J_j \rangle = 0, \\
    &\dot{f}^i f^j\langle R(J_j, \dot{q})D_t \dot{q}, J_i\rangle = 0.
\end{align*}
Hence we may add these terms into our expansion to utilize Lemma \ref{Bijacobi symmetry}. That is,
\begin{align*}
    I(X, X) &= I(J, J) + \int_a^b \|\dot{f}^i D_t J_i \|^2dt + \dot{f}^i f^j P_{-}(J_i, J_j). 
\end{align*}
However, $P_{-}(\cdot, \cdot)$ is a tensor field, and so $\dot{f}^i f^j P_{-}(J_i, J_j) = P_{-}(\dot{f}^i J_i, f^j J_j) = 0$ since $\dot{f}^i J_i = 0$. Therefore, $I(X, X) - I(J, J) = \int_a^b \|\dot{f}^i D_t J_i \|^2 dt \ge 0.$ Moreover, if $\int_a^b \|\dot{f}^i D_t J_i \|^2 dt = 0$, it follows immediately that $\dot{f}^i D_t J_i = 0$. Since it also holds that $\dot{f}^i J_i = 0$, and $(J_i, D_t J_i)$ is a basis for $T_q Q \times T_q Q$, it must be the case that $\dot{f}^i(t) = 0$ for all $t \in [a, b]$, $i = 1, \dots, 2n$. As $f^i(b) = c^i$ for all $i = 1,\dots, 2n$, it follows that, in fact, $f^i(t) = c^i$ for all $t \in [a, b], i = 1, \dots, 2n$. Therefore, $X \equiv J$ on $[a,b]$.
\end{proof}

Theorems \ref{thm: not_min} and \ref{thm: suff_conds_biconjugate} can be combined into the following corollary:

\begin{corollary}\label{cor: suff_cond}
A modified Riemannian cubic $q \in \Omega$ is an $\Omega$-local minimizer of $J$ if and only if there are no biconjugate points along $q.$
\end{corollary}

\section{Reduction of Sufficient Conditions for variational obstacle avoidance}\label{section: reduction suff conds Lie groups}

Next, we apply reduction by symmetry to the sufficient conditions for optimality on a Lie group equipped with a left-invariant metric. In the end, this amounts to left-translating the Bi-Jacobi fields described by equation \eqref{eq: bi-Jacobi field} to the Lie algebra $\g$, and studying the corresponding bi-conjugate points so that we may apply Corollary \ref{cor: suff_cond}.

To that end, suppose that $G$ is a Lie group satisfying that it is a connected Lie group endowed with a left-invariant Riemannian metric and corresponding Levi-Civita connection $\nabla$, and $g \in \Omega$ is a modified Riemannian cubic. Let $\xi^{(0)} := g^{-1} \dot{g}$, and recursively define $\xi^{(i+1)} = \dot{\xi}^{(i)} + \nabla^\g_{\xi^{(0)}} \xi^{(i)}$ for $i = 0, 1, 2$. Consider a vector field $X \in \Gamma(TG)$. From Lemma \ref{lemma: cov-to-covg}, it is clear that if we recursively define
$\mathcal{X}^{(i+1)} := \dot{\mathcal{X}}^{(i)} + \nabla^\g_{\xi^{(0)}} \mathcal{X}^{(i)}$ with $\mathcal{X}^{(0)} := g^{-1} X$, then $g^{-1} D_t^i X = \mathcal{X}^{(i)}$ for all $i \in \N$. In particular,

\begin{equation}\label{eq: Dt4X Lie algebra}
    g^{-1} D_t^4 X = \dot{\mathcal{X}}^{(3)} + \nabla^\g_{\xi^{(0)}} \mathcal{X}^{(3)}.  
\end{equation}

\noindent We now translate each term of $F(X, \dot{g}),$ defined in \eqref{eq: F}. Since the Riemannian curvature $R$ is a tensor field, it follows that
\begin{align}
    R(R(X, \dot{g})\dot{g}, \dot{g})\dot{g} &= gR(R(\mathcal{X}^{(0)}, \xi^{(0)})\xi^{(0)}, \xi^{(0)})\xi^{(0)} \label{eq: F translation 1} \\
    R(X, D_t^2 \dot{g})\dot{g} &= gR(\mathcal{X}^{(0)}, \xi^{(2)})\xi^{(0)} \label{eq: F translation 2}\\
    R(D_t^2 X, \dot{g})\dot{g} &= gR(\mathcal{X}^{(2)}, \xi^{(0)})\xi^{(0)} \label{eq: F translation 3}\\
    R(X, \dot{g})D_t^2 \dot{g} &= gR(\mathcal{X}^{(0)}, \xi^{(0)})\xi^{(2)} \label{eq: F translation 4}\\
    R(X, D_t \dot{g})D_t \dot{g} &= gR(\mathcal{X}^{(0)}, \mathcal{X}^{(1)})\mathcal{X}^{(1)} \label{eq: F translation 5}\\
    R(D_t X, \dot{g})D_t \dot{g} &= gR(\mathcal{X}^{(1)}, \xi^{(0)})\xi^{(1)}.\label{eq: F translation 6}
\end{align}
The remaining terms of $F(X, \dot{q})$ involve covariant derivatives of the Riemannian curvature. Note that for $X, Y, Z \in \Gamma(TG)$, we have
\begin{align*}
    (D_t R)(X, Y)Z &= D_t (R(X, Y)Z) - R(D_t X, Y)Z \\&- R(X, D_t Y)Z - R(X, Y) D_t Z \\
    &= D_t (gR(\mathcal{X}, \mathcal{Y})\mathcal{Z}) - gR(\dot{\mathcal{X}} + \nabla^\g_\xi \mathcal{X}, \mathcal{Y})\mathcal{Z} \\&- gR(\mathcal{X}, \dot{\mathcal{Y}} + \nabla^\g_\xi \mathcal{Y})\mathcal{Z} - gR(\mathcal{X}, \mathcal{Y})(\dot{\mathcal{Z}} + \nabla^\g_\xi \mathcal{Z}) \\
    &= D_t (gR(\mathcal{X}, \mathcal{Y})\mathcal{Z})\\& - g\left(\frac{d}{dt} R(\mathcal{X}, \mathcal{Y})\mathcal{Z} - R(\nabla^\g_\xi \mathcal{X}, \mathcal{Y})\mathcal{Z}\right.\\& \left.- R( \mathcal{X}, \nabla^\g_\xi\mathcal{Y})\mathcal{Z} - R( \mathcal{X}, \mathcal{Y})\nabla^\g_\xi \mathcal{Z}\right). \\
\end{align*}
Moreover,
\begin{align*}
    D_t (gR(\mathcal{X}, \mathcal{Y})\mathcal{Z}) &= g\left(\frac{d}{dt} R(\mathcal{X}, \mathcal{Y})\mathcal{Z} + \nabla_\xi^\g
    \left(R(\mathcal{X}, \mathcal{Y})\mathcal{Z}\right) \right).
\end{align*}
Hence, if we let $\mathcal{X} = g^{-1} X, \ \mathcal{Y} = g^{-1} Y, \ \mathcal{Z} = g^{-1} Z$, we find that
\begin{align}
    (D_t R)(X, Y)Z &= g\left( \nabla_\xi^\g (R(\mathcal{X}, \mathcal{Y})\mathcal{Z}) - R(\nabla^\g_\xi \mathcal{X}, \mathcal{Y})\mathcal{Z}\right.\nonumber\\ &\left.- R( \mathcal{X}, \nabla^\g_\xi\mathcal{Y})\mathcal{Z} - R( \mathcal{X}, \mathcal{Y})\nabla^\g_\xi \mathcal{Z}\right).\label{Reduced_Derivative_Curvature}
\end{align}
Therefore, 
\begin{align}
    g^{-1}(D_t R)(D_t X, \dot{g})\dot{g} &= \nabla^\g_{\xi^{(0)}} \left(R(\mathcal{X}^{(1)}, \xi^{(0)})\xi^{(0)}\right)\nonumber\\ &- R(\nabla^\g_{\xi^{(0)}} \mathcal{X}^{(1)}, \xi^{(0)})\xi^{(0)} \\
    &\quad - R(\mathcal{X}^{(1)}, \nabla^\g_{\xi^{(0)}}\xi^{(0)})\xi^{(0)} \nonumber\\&- R(\mathcal{X}^{(1)}, \xi^{(0)}) \nabla^\g_{\xi^{(0)}}\xi^{(0)}, \nonumber\\
    g^{-1}(D_t R)(X, D_t \dot{g})\dot{g} &= \nabla^\g_{\xi^{(0)}} \left(R(\mathcal{X}^{(0)}, \xi^{(1)})\xi^{(0)}\right) \nonumber\\&- R(\nabla^\g_{\xi^{(0)}} \mathcal{X}^{(0)}, \xi^{(1)})\xi^{(0)} \\ 
    &- R(\mathcal{X}^{(0)}, \nabla^\g_{\xi^{(1)}}\xi^{(0)})\xi^{(0)}\nonumber\\ &- R(\mathcal{X}^{(0)}, \xi^{(1)}) \nabla^\g_{\xi^{(0)}}\xi^{(0)}, \nonumber\\
    g^{-1}(D_t R)(X, \dot{g})D_t \dot{g} &= \nabla^\g_{\xi^{(0)}} \left(R(\mathcal{X}^{(0)}, \xi^{(0)})\xi^{(1)}\right)\nonumber\\ &- R(\nabla^\g_{\xi^{(0)}} \mathcal{X}^{(0)}, \xi^{(0)})\xi^{(1)} \\ 
    &- R(\mathcal{X}^{(0)}, \nabla^\g_{\xi^{(0)}}\xi^{(0)})\xi^{(1)}\nonumber\\ &- R(\mathcal{X}^{(0)}, \xi^{(0)}) \nabla^\g_{\xi^{(0)}}\xi^{(1)} \nonumber. 
\end{align}
\normalsize By a similar argument,
\begin{align*}
    g^{-1}(D_t^2 R)(X, Y)Z &= \nabla_{\xi^{(0)}}^\g ((D_t R)(\mathcal{X}, \mathcal{Y})\mathcal{Z})\\& - (D_t R)(\nabla^\g_{\xi^{(0)}} \mathcal{X}, \mathcal{Y})\mathcal{Z}\\& - (D_t R)( \mathcal{X}, \nabla^\g_{\xi^{(0)}}\mathcal{Y})\mathcal{Z}\\& - (D_t R)( \mathcal{X}, \mathcal{Y})\nabla^\g_{\xi^{(0)}} \mathcal{Z},
\end{align*}
\normalsize which may then be used to calculate $g^{-1}(D_t^2 R)(X, \dot{g})\dot{g}$ as a function of $\mathcal{X}^{(0)}$ and $\xi^{(0)}$ and written in terms of the Riemannian curvature $R$ by applying \eqref{Reduced_Derivative_Curvature} to each term.

The remaining term of $F(X, \dot{q})$ is $(\nabla_X R)(D_t \dot{g}, \dot{g})\dot{g}$, which differs from the rest because we are now taking the covariant derivative of the curvature with respect to $X$. To address this, suppose that $\Gamma(s, t)$ is a two-parameter variation of $g$, and define $T = \partial_t \Gamma(s,t)$ and $S := \partial_s \Gamma(s,t)\big{\vert}$. Further define $\mathcal{T}(s, t) := \Gamma^{-1} T$ and $\mathcal{S}(s, t) := \Gamma^{-1} S$, and suppose $\Gamma$ is defined such that $T(0, t) = \dot{g}(t)$ and $S(0, t) = X(t).$ Then, it can be seen by applying equation \eqref{Reduced_Derivative_Curvature} that:
\begin{align*}
    (D_s R)(D_t T, T)T &= D_s (R(D_t T, T)T) - R(D_s D_t T, T)T \\&- R(D_t T, D_s T)T - R(D_t T, T)D_s T \\
    &= \Gamma\big{[}\nabla^\g_{\mathcal{S}} \big{(}R(\dot{\Tg} + \nabla^\g_{\Tg} \Tg, \Tg)\Tg\big{)}\\&- R(\nabla^\g_\Sg (\nabla^\g_\Tg + \dot{\Tg}), \Tg)\Tg \\&- R(\nabla^\g_\Tg + \dot{\Tg}, \nabla^\g_\Sg \Tg)\Tg\\& - R(\nabla^\g_\Tg + \dot{\Tg},  \Tg)\nabla^\g_\Sg\Tg \big{]}
\end{align*}
\normalsize
Setting $s = 0,$ we obtain
\begin{align}
    g^{-1}(\nabla_X R)(D_t \dot{g}, \dot{g})\dot{g} &= \nabla_{\mathcal{X}^{(0)}}(R(\xi^{(1)}, \xi^{(0)})\xi^{(0)})\nonumber\\& - R(\nabla^\g_{\mathcal{X}^{(0)}} \xi^{(1)}, \xi^{(0)})\xi^{(0)}\label{Reduced_X_Derivative_Curvature} \\
    &\quad - R( \xi^{(1)}, \nabla^\g_{\mathcal{X}^{(0)}}\xi^{(0)})\xi^{(0)} \\&- R( \xi^{(1)}, \xi^{(0)})\nabla^\g_{\mathcal{X}^{(0)}}\xi^{(0)}.\nonumber
\end{align}
\normalsize From the previous analysis, it is clear that $F(X, \dot{q})$ satisfies
\begin{equation*}
    g^{-1} F(X, \dot{q}) = \mathcal{F}(\mathcal{X}^{(0)}, \mathcal{X}^{(1)}, \mathcal{X}^{(2)}, \xi^{(0)}, \xi^{(1)}, \xi^{(2)}),
\end{equation*} for some smooth multilinear function $\mathcal{F}: \g^6 \to \g$.

Finally, we must translate $\nabla_X \grad V(g)$ to $\g$. To that end, let $\alpha$ be a variation of $g$ such that $\partial_s \alpha \vert_{s=0} = X.$ Moreover, let $\{A_i\}$ be a basis for $\mathfrak{g}$ and write $g_0 \grad_1 V_{\ext}(g, g_0) = V^i(g, g_0) gA_i$ for $V^i: G \times G \to \R$ for $i = 1, \dots, \dim(G)$, we have
\begin{small}\begin{align*}
    \nabla_X \grad V(g) &= D_s \grad V(\alpha) \big{\vert}_{s=0} \\
    &= D_s \grad_1 V_{\ext}(\alpha, g_0) \big{\vert}_{s=0} \\
    &= D_s \left(V^i(\alpha, g_0) A_i\right) \big{\vert}_{s=0} \\
    &= g\left(\frac{\partial}{\partial s} V^i(\alpha, g_0) \Big\vert_{s=0} A_i + \nabla^\g_{\mathcal{X}} g^{-1} \grad_1 V_{\ext}(g, g_0) \right).
\end{align*}\end{small}

Observe that, due to the symmetry of the extended potential,
\begin{align*}
    \langle \grad_1 V_{\ext}(g, g_0), X\rangle &= \frac{\partial}{\partial s}\Big{\vert}_{s=0} V_{\ext}(\alpha, g_0) \\&= \frac{\partial}{\partial s}\Big{\vert}_{s=0}  V_{\ext}(g^{-1}\alpha, h) \\&= \left\langle g\grad_1 V_{\ext}(e, h), X \right\rangle,
\end{align*}
so that $g^{-1}\grad_1 V_{\ext}(g, g_0) = \grad_1 V_{\ext}(e, h)$. Moreover, this implies that $V^i(\alpha, g_0) = V^i(g^{-1} \alpha, h)$, so that \begin{align*}\frac{\partial}{\partial s} V^i(\alpha(s), g_0) \Big{\vert}_{s=0} A_i &= \frac{\partial}{\partial s} V^i(g^{-1}\alpha(s), h) \Big{\vert}_{s=0} A_i \\&= d_1 V^i_{(e, h)}(\mathcal{X})A_i.\end{align*} Furthermore, this expression is independent of the chosen basis for $\g$, which allows to define a linear operator $D_{\mathcal{X}}: \g \to \g$ such that $D_{\mathcal{X}} \grad_1 V_{\ext}(e, h) = d_1 V^i_{(e, h)}(\mathcal{X})A_i$ when written with respect to any basis $\{A_i\}$ for $\g$. In all, we see that
\begin{align}
    \nabla_X \grad V(g) = g(D_{\mathcal{X}^{(0)}} + \nabla^\g_{\mathcal{X}^{(0)}}) \grad_1 V_{\ext}(e, h).
\end{align}
Hence, under the assumption that $G$ is a connected Lie group satisfying endowed with a left-invariant Riemannian metric and corresponding Levi-Civita connection $\nabla$, we have proven the following result:
\begin{theorem}\label{prop: reduction_G_left_inv_suff_cond}
Suppose that $G$ satisfies the previous assumption, and let $g \in \Omega$ solve \eqref{eq: necessary}. Define $\xi^{(i+1)} = \dot{\xi}^{(i)} + \nabla_{\xi^{(0)}}^\g \xi^{(i)}$ for $i = 0, 1, 2$, with $\xi^{(0)} := g^{-1} \dot{g}$, and let $h := g^{-1}g_0$ for some $g_0 \in G$. Then $X$ is a bi-Jacobi field along $g$ if and only if $\mathcal{X}^{(0)} := g^{-1} X$ solves:
\begin{align}
    \mathcal{X}^{(i+1)} =& \dot{\mathcal{X}}^{(i)} + \nabla^{\g}_{\xi^{(0)}} \mathcal{X}^{(i)}, \qquad \text{ for } i = 0, 1, 2, \label{eq: reduced_bi-Jacobi constraints}\\
    0=&\dot{\mathcal{X}}^{(3)} + \nabla^\g_{\xi^{(0)}} \mathcal{X}^{(3)}+ \left(D_{\mathcal{X}^{(0)}} + \nabla^\g_{\mathcal{X}^{(0)}}\right) \grad_1 V_{\ext}(e, h) \nonumber\\& + \mathcal{F}(\mathcal{X}^{(0)}, \mathcal{X}^{(1)}, \mathcal{X}^{(2)},\xi^{(0)}, \xi^{(1)}, \xi^{(2)}),\label{eq: reduced_bi-jacobi ODE}
\end{align}
We call a smooth solution $\mathcal{X}^{(0)}$ to \eqref{eq: reduced_bi-Jacobi constraints}-\eqref{eq: reduced_bi-jacobi ODE} a \textit{reduced bi-Jacobi field}.
\end{theorem}

Note that, contrary to the case of modified Riemannian cubics, where the reduction process reduces the order of the governing ODE by $1$, the resulting reduced equations in Theorem \ref{prop: reduction_G_left_inv_suff_cond} are of the same order as the original equation describing bi-Jacobi fields. Ultimately, equations \eqref{eq: reduced_bi-Jacobi constraints}-\eqref{eq: reduced_bi-jacobi ODE} are simply a translation of \eqref{eq: bi-Jacobi field} to the Lie algebra $\g$. Despite not reducing the order, there are numerous advantage to this. As we will see in Proposition \ref{prop: reduction_suff_conds}, the vector space structure of $\g$ allows us to express the sufficient conditions for optimality (as outlined in Corollary \ref{cor: suff_cond}) succinctly in terms of the determinant of a matrix which depends only on the solutions to a series of initial value problems. Moreover, in many applications where $G$ admits a bi-invariant metric, the curvature tensor $R$ (and thus $\mathcal{F})$ may be calculated explicitly on $\g$, in addition to many obstacle avoidance artificial potentials, which greatly simplifies the solution of the initial value problems.

Observe that left-translation provides an isomorphism $T_g\Omega \cong \mathring{H}^2([a,b], \g)$, where $\mathring{H}^2([a,b], \g)$ denotes the space of Sobolev class $H^2$ curves $\eta: [a, b] \to \g$ such that
\begin{equation*}
    \eta(a) = \eta(b) = 0,\,\dot{\eta}(a) = -\nabla^\g_{\xi^{(0)}(a)} \eta(a),\, \dot{\eta}(b) = -\nabla^\g_{\xi^{(0)}(b)} \eta(b),
\end{equation*}
where $\xi := g^{-1} \dot{g}$. Hence, we are able to reinterpret the index form \eqref{eq: index form} along a modified Riemannian cubic $g \in \Omega$ as the bilinear form $\mathcal{I}: \mathring{H}^2([a,b], \g) \times \mathring{H}^2([a,b], \g) \to \R$ defined by:
\begin{align}
    &\mathcal{I}(\mathcal{X}^{(0)}, \mathcal{Y}^{(0)}) = \int_a^b [\left\langle \mathcal{X}^{(2)}, \mathcal{Y}^{(2)} \right\rangle\label{eq: reduced index form}\\
    & + \left\langle \mathcal{Y}^{(0)}, \ \mathcal{F}(\mathcal{X}, \xi) + \left(D_{\mathcal{X}^{(0)}} + \nabla^\g_{\mathcal{X}^{(0)}}\right)\grad_1 V_{\ext}(e, g^{-1}g_0)\right\rangle]dt,\nonumber
\end{align}
\normalsize where we have used the notation $\xi = (\xi^{(1)}, \xi^{(2)}, \xi^{(3)})$ and $\mathcal{X} = (\mathcal{X}^{(0)}, \mathcal{X}^{(1)}, \mathcal{X}^{(2)})$, and recursively defined $\xi^{(i+1)} = \dot{\xi}^{(i)} + \nabla_{\xi^{(0)}}^\g \xi^{(i)}$ and $\mathcal{X}^{(i+1)} = \dot{\mathcal{X}}^{(i)} + \nabla^{\g}_{\xi^{(0)}} \mathcal{X}^{(i)}$ for $i = 0, 1, 2$, with $\xi^{(0)} := g^{-1} \dot{g}$. We call $\mathcal{I}$ the \textit{reduced index form}. It is clear that $\mathcal{I}$ is equivalent to $I$, in the sense that for all $X, Y \in T_g \Omega$, we have that $\mathcal{X} = g^{-1} X$ and $\mathcal{Y} = g^{-1} Y$ satisfy $I(X, Y) = \mathcal{I}(\mathcal{X}, \mathcal{X})$, and vice versa. In particular, we have $\ker(\mathcal{I}) = \{g^{-1} X \in \mathring{H}^2([a, b], \g) \ | \ X \in T_g \Omega$ is a bi-Jacobi field $\}$. From Theorem \ref{prop: reduction_G_left_inv_suff_cond}, it follows that the kernel of $\mathcal{I}$ is precisely the set of reduced bi-Jacobi fields. Moreover, $t=t_0$ and $t=t_1$ are biconjugate along $g$ if and only if there exists a reduced bi-Jacobi field satisfying $\mathcal{X}^{(0)}(t_0) = \mathcal{X}^{(0)}(t_1) = \mathcal{X}^{(1)}(t_0) = \mathcal{X}^{(1)}(t_1) = 0$. This leads to the following proposition:

\begin{proposition}\label{prop: reduction_suff_conds}
Let $\{A_i\}$ be a basis for $\g$, and suppose that $\mathcal{X}_i^{(0)}$ is a reduced bi-Jacobi fields satisfying the initial conditions \text{for } $i=1,\dots, n$,
\begin{align*}
    &\mathcal{X}_i^{(0)}(a) = 0, \quad \mathcal{X}_i^{(1)}(a) = 0, \quad \mathcal{X}_i^{(2)}(a) = A_i, \quad \mathcal{X}_i^{(3)}(a) = 0, \end{align*} and 
    \begin{align*}
    &\mathcal{X}_i^{(0)}(a) = 0, \quad \mathcal{X}_i^{(1)}(a) = 0, \quad \mathcal{X}_i^{(2)}(a) = 0, \quad \  \mathcal{X}_i^{(3)}(a) = A_{i - n}
\end{align*}\text{for } $i=n+1,\dots, 2n$.

Let $\alpha_i^k, \beta_i^k \in \R$ be such that $\mathcal{X}_i^{(0)}(t) = \alpha^k_i(t) A_k$ and $\mathcal{X}_i^{(1)}(t) = \beta^k_i(t) A_k$ for $i = 1,\dots, 2n$ and define $A(t) = \begin{bmatrix} \alpha_1^1 & \dots & \alpha_{1}^{n} & \beta_{1}^1 & \dots & \beta_1^{n}\\ \vdots & \ddots & \vdots & \vdots & \ddots & \vdots \\ \alpha_{2n}^1 & \dots & \alpha_{2n}^{n} & \beta^1_{2n} & \dots & \beta^{n}_{2n}  \end{bmatrix}.$
Then $g$ is an $\Omega$-local minimizer of $J$ if and only if $\det(A(t)) \ne 0$ for all $t \in (a, b]$.
\end{proposition}
\begin{proof}
Let $X_i := g \mathcal{X}_i^{(0)}$ for $i=1,\dots, 2n$. From Theorem \ref{prop: reduction_G_left_inv_suff_cond}, it is clear that each $X_i$ is a Bi-Jacobi field. Moreover, since $\{A_i\}$ is a basis for $\g$, $\{X_i\}_{i=1}^{2n}$ is a basis for the space $J_g(0)$ of Bi-Jacobi fields along $g$ which vanish at $t =a$ along with their first covariant derivatives. Hence, for all $Z \in J_g(0)$, there exists $2n$ constants $a^i \in \R$ such that 
\begin{align*}
    Z &= \sum_{i=1}^{2n} a^i X_i = g\sum_{i=1}^{2n} a^i \mathcal{X}^{(0)}_i, \\
    D_t Z &= \sum_{i=1}^{2n} a^i D_t X_i = g\sum_{i=1}^{2n} a^i \mathcal{X}^{(1)}_i,
\end{align*}
for all $t \in [a, b]$. In particular, this implies that $$\displaystyle{(Z, D_t Z) = g \sum_{i=1}^{2n} a^i (\mathcal{X}^{(0)}_i, \mathcal{X}^{(1)}_i)},$$ where we define the left-action $G \times (\g \times \g) \to \g \times \g$ by $(g, (\xi, \eta)) \mapsto (L_{g^\ast} \xi, L_{g^\ast} \eta)$. 

Let $t_0 \in (a, b]$ and suppose that $g$ is an $\Omega$-local minimizer of $J$. Then the only bi-Jacobi field $Z \in J_g(0)$ such that $Z(t_0) = D_t Z(t_0) = 0$ is the zero vector field $Z \equiv 0$ by Corollary \ref{cor: suff_cond}. This implies that $\displaystyle{\sum_{i=1}^{2n} a^i (\mathcal{X}^{(0)}_i(t_0), \mathcal{X}^{(1)}_i(t_0))} = 0$ if and only if $a^i = 0$ for $i = 1,\dots, 2n$. It follows immediately that $\left\{\big{(}\mathcal{X}^{(0)}_i(t_0), \mathcal{X}^{(1)}_i(t_0)\big{)}\right\}_{i=1}^{2n}$ is a basis for $\g \times \g$, which implies that $\det(A(t_0)) \ne 0$. Since this holds for all $t_0 \in (a, b]$, the result holds. 

Now suppose that $g$ is not an $\Omega$-local minimizer. Then there is a time $t = t_0 \in (a, b]$ which is biconjugate to $t=a$ by Corollary \ref{cor: suff_cond}. Hence, by definition, there exists a non-trivial bi-Jacobi field $Z \in J_g(0)$ such that $Z(t_0) = D_t Z(t_0) = 0$, and so there is a non-trivial solution to $\sum_{i=1}^{2n} a^i (\mathcal{X}^{(0)}_i(t_0), \mathcal{X}^{(1)}_i(t_0))  = 0$. In particular, the vectors $(\mathcal{X}^{(0)}_i(t_0), \mathcal{X}^{(1)}_i(t_0))$ are linearly dependent, so that $\det(A(t_0)) = 0$.
\end{proof}

\subsection{The Reduced Obstacle Avoidance Problem}

Suppose that $G$ is a connected Lie group equipped with a left-invariant Riemannian metric $\langle \cdot, \cdot \rangle$. We fix some $g_0 \in G$, which we consider a point-obstacle, and choose an artificial potential of the form $V(g) = f(d^2(g, g_0)),$ where $f: \R \to \R$ is smooth and non-negative, and $d^2(g, g_0)$ refers to the square of the Riemannian distance on $G$ with respect to $\left< \cdot, \cdot \right>.$ The extended potential then takes the form $V_{\ext}(g, g_0) = f(d^2(g, g_0))$, which satisfies the required symmetries given that the Riemannian distance with respect to a left-invariant metric is itself left-invariant. That is, $d(hg, hg_0) = d(g, g_0)$ for all $g, g_0, h \in G.$ Moreover, the gradient vector field of the potential is given by $\grad_1 V_{\ext}(e, h) = f'(d^2(e, h)) \grad_1 d^2(e, h).$ This form becomes more tractable under the assumption that $h(t)$ is contained within a geodesically convex neighborhood of $e$ for all $t \in [0, T],$ as in such a case we have $d(e, h) = \|\exp_e^{-1}(h)\|,$ where $\exp$ is the Riemannian exponential map. It was shown in \cite{goodman2022reduction}that $\grad_1 d^2(e, h) = 2\exp_e^{-1}(h)$, so that $\grad_1 V_{\ext}(e, h) = 2f'(\|\exp_e^{-1}(h)\|^2)\exp_e^{-1}(h).$

This term may be simplified considerably in the case that $G$ admits a bi-invariant metric $\left< \cdot, \cdot\right>_{\Bi}.$ Suppose that $\beta: \g \to \g$ is the linear endomorphism such that $\left<\xi, \eta\right>_{\Bi} = \left< \beta(\xi), \eta\right>$ for all $\xi, \eta \in \g$. Let $\xi \in \g$, and observe that
$$\left<\grad_1 V_{\ext}(e,h), \xi \right> =  \left<\beta\left(\grad_1^{\Bi} V_{\ext}(e,h)\right), \xi \right>,$$
where $\grad^{\Bi}_1 V_{\ext}$ denotes the gradient vector field of the extended artificial potential $V$ with respect to its first component and the metric $\left< \cdot, \cdot\right>_{\Bi}.$ We now suppose that the potential and the corresponding extended potential take the form $V(g) = V_{\ext}(g, g_0) = f(d^2_{\Bi}(g, g_0))$, where now $d^2_{\Bi}: G \times G \to \R$ is the Riemannian distance with respect to $\left< \cdot, \cdot \right>_{\Bi}.$ Following the previous analysis, we find that $\grad_1 V_{\ext}(e, h) = 2f'(\|(\exp^{\Bi})_e^{-1}(h)\|^2)(\exp^{\Bi})_e^{-1}(h)$ as long as $h$ is contained in a geodesically convex neighborhood of $e$, where $\exp^{\Bi}$ is the Riemannian exponential map with respect to $\left< \cdot, \cdot \right>_{\Bi}.$ As shown in \cite{bullo2019geometric}, $\exp^{\Bi}_e = \Exp$ and $(\exp^{\Bi})^{-1}_e = \Log$, where $\Exp$ and $\Log$ are respectively the Lie exponential map and the logarithmic map on $G.$ Hence, equation \eqref{eq: reduced_bi-jacobi ODE} takes the form 

\begin{align*}
    0&=\dot{\mathcal{X}}^{(3)} + \mathcal{F}(\mathcal{X}^{(0)}, \mathcal{X}^{(1)}, \mathcal{X}^{(2)}, \xi^{(0)}, \xi^{(1)}, \xi^{(2)})\nonumber\\& + 2\left(D_{\mathcal{X}^{(0)}} + \nabla^\g_{\mathcal{X}^{(0)}}\right) \beta\left(f'(\|\Log(h)\|^2)\Log(h)\right).\label{eq: reduced_bi-jacobi ODE left-bi-inv}
\end{align*}

Observe that the remaining components of equations \eqref{eq: reduced_bi-Jacobi constraints}-\eqref{eq: reduced_bi-jacobi ODE} are still written with respect to the left-invariant metric on $G$. This situation arises naturally for rigid body motion on $\SO(3)$, as the natural metric (corresponding to the kinetic energy of a rigid body) is given by $\left< \dot{R}_1, \dot{R}_2\right> = \tr(\dot{R}_1 \mathbb{M} \dot{R}_2^T),$ where $\dot{R}_1, \dot{R}_2 \in \Gamma(G)$ and $\mathbb{M}$ is a symmetric positive-definite $3 \times 3$ matrix called the \textit{coefficient of inertia matrix}. In such a case, the metric is left-invariant, and it is bi-invariant if and only if $\mathbb{M} = I$, the $3\times 3$ identity matrix—which occurs only for perfectly symmetric rigid bodies. Hence, despite $\SO(3)$ admitting a bi-invariant metric, we are forced to use a left-invariant metric for $J$ and when defining the Levi-Civita connection and Riemannian curvature. However, we are still free to define the artificial potential $V$ with respect to the bi-invariant metric, as it is not derived from the physical situation and thus depends only on the Lie group $\SO(3).$ The principal advantage of this is that in many situations (such as $G = \SO(3)$), the logarithmic map may be calculated explicitly, whereas the exponential map with respect to the left-invariant metric typically cannot. 

In the case that our left-invariant metric is bi-invariant (that is, where $\beta$ is the identity map), we further have the identities
\begin{align*}
    \nabla_{\xi}^\g \eta = \frac12 \left[\xi, \eta\right]_{\g}, \,\,
    R\big{(}\xi, \eta{)}\sigma = -\frac14 \big{[}\big{[}\xi, \eta \big{]}_{\g}, \sigma \big{]}_\g,
\end{align*}
for all $\xi, \eta, \sigma \in \g,$ which allows for even further simplifications.

\bibliography{ifacconf}            % bib file to produce the bibliography

\begin{thebibliography}{22}
\providecommand{\natexlab}[1]{#1}
\providecommand{\url}[1]{\texttt{#1}}
\providecommand{\urlprefix}{URL }
\expandafter\ifx\csname urlstyle\endcsname\relax
  \providecommand{\doi}[1]{doi:\discretionary{}{}{}#1}\else
  \providecommand{\doi}{doi:\discretionary{}{}{}\begingroup \urlstyle{rm}\Url}\fi

\bibitem[{Assif et~al.(2018)Assif, Banavar, Bloch, Camarinha, and Colombo}]{mishal}
Assif, M., Banavar, R., Bloch, A., Camarinha, M., and Colombo, L.J. (2018).
\newblock Variational collision avoidance problems on riemannian manifolds.
\newblock \emph{Proceedings of the 2018 IEEE International Conference on Decision and Control}, 2791--2796.

\bibitem[{Bloch et~al.(2017)Bloch, Camarinha, and Colombo}]{BlCaCoCDC}
Bloch, A., Camarinha, M., and Colombo, L.J. (2017).
\newblock Variational obstacle avoidance on riemannian manifolds.
\newblock \emph{Proceedings of the 2017 IEEE International Conference on Decision and Control}, 146--150.

\bibitem[{Bloch et~al.(2021{\natexlab{a}})Bloch, Camarinha, and Colombo}]{BlCaCoIJC}
Bloch, A., Camarinha, M., and Colombo, L.J. (2021{\natexlab{a}}).
\newblock Dynamic interpolation for obstacle avoidance on riemannian manifolds.
\newblock \emph{International Journal of Control}, 94(3), 588--600.

\bibitem[{Bloch et~al.(2021{\natexlab{b}})Bloch, Camarinha, and Colombo}]{point}
Bloch, A., Camarinha, M., and Colombo, L.J. (2021{\natexlab{b}}).
\newblock Variational point-obstacle avoidance on riemannian manifolds.
\newblock \emph{Mathematics of Control, Signals, and Systems}, 33, 109--121.

\bibitem[{Boothby(2003)}]{Boothby}
Boothby, W.M. (2003).
\newblock \emph{An introduction to differentiable manifolds and Riemannian geometry, Revised}, volume 120.
\newblock Gulf Professional Publishing.

\bibitem[{Bullo and Lewis(2019)}]{bullo2019geometric}
Bullo, F. and Lewis, A.D. (2019).
\newblock \emph{Geometric control of mechanical systems: modeling, analysis, and design for simple mechanical control systems}, volume~49.
\newblock Springer.

\bibitem[{Camarinha et~al.(1995)Camarinha, Leite, and Crouch}]{marg}
Camarinha, M., Leite, F.S., and Crouch, P. (1995).
\newblock Splines of class $c^k$ on non-euclidean spaces.
\newblock \emph{IMA Journal of Mathematical Control \& Information}, 12, 299--410.

\bibitem[{Camarinha et~al.(2001)Camarinha, Leite, and Crouch}]{Sufficient2001}
Camarinha, M., Leite, F.S., and Crouch, P. (2001).
\newblock On the geometry of riemannian cubic polynomials.
\newblock \emph{Differential Geometry and its Applications}, 15, 107--135.

\bibitem[{Chandrasekaran et~al.(2020)Chandrasekaran, Colombo, Camarinha, Banavar, and Bloch}]{sh}
Chandrasekaran, R., Colombo, L.J., Camarinha, M., Banavar, R., and Bloch, A. (2020).
\newblock Variational collision and obstacle avoidance of multi-agent systems on riemannian manifolds.
\newblock \emph{Proceedings of the 2020 European Control Conference}.

\bibitem[{Colombo and Goodman(2020)}]{CoGo20}
Colombo, L. and Goodman, J. (2020).
\newblock A decentralized strategy for variational collision avoidance on complete riemannian manifolds.
\newblock \emph{Proceedings of the 2020 Portuguese Conference on Automatic Control}, 363--372.

\bibitem[{Colombo and Goodman(2023)}]{colombo2023existence}
Colombo, L. and Goodman, J. (2023).
\newblock Existence of global minimizer for elastic variational obstacle avoidance problems on riemannian manifolds.
\newblock In \emph{International Conference on Geometric Science of Information}, 81--88. Springer.

\bibitem[{Crouch and Leite(1991)}]{CLACC}
Crouch, P. and Leite, F.S. (1991).
\newblock Geometry and the dynamic interpolation problem.
\newblock \emph{Proceedings of the 1991 American Control Conference}, 16(4), 1131--1137.

\bibitem[{Crouch and Leite(1995)}]{CroSil:95}
Crouch, P. and Leite, F.S. (1995).
\newblock The dynamic interpolation problem: on riemannian manifolds, lie groups, and symmetric spaces.
\newblock \emph{Journal of Dynamical and Control Systems}, 1, 177--202.

\bibitem[{Giambò et~al.(2002)Giambò, Giannoni, and Piccione}]{Giambo}
Giambò, R., Giannoni, F., and Piccione, P. (2002).
\newblock An analytical theory for riemannian cubic polynomials.
\newblock \emph{IMA Journal of Math, Control, and Information}, 19(4), 445--460.

\bibitem[{Giambò et~al.(2004)Giambò, Giannoni, and Piccione}]{RiemannianPoly}
Giambò, R., Giannoni, F., and Piccione, P. (2004).
\newblock Optimal control on riemannian manifolds by interpolation.
\newblock \emph{Mathematics of Control, Signal and Systems}, 16(4), 278--296.

\bibitem[{Goodman(2023)}]{Goodmanthesis}
Goodman, J. (2023).
\newblock \emph{Path Planning on Riemannian Manifolds with Applications to Quadrotor Load Transportation}.
\newblock Ph.D thesis, Universidad Autónoma de Madrid.

\bibitem[{Goodman(2022)}]{goodman2022sufficient}
Goodman, J.R. (2022).
\newblock Local minimizers for variational obstacle avoidance on riemannian manifolds.
\newblock \emph{Journal of Geometric Mechanics}, 15(1), 59--72.

\bibitem[{Goodman and Colombo(2021)}]{goodman2021obstacle}
Goodman, J.R. and Colombo, L.J. (2021).
\newblock Variational obstacle avoidance with applications to interpolation problems in hybrid systems.
\newblock \emph{IFAC-PapersOnLine}, 54(19), 82--87.

\bibitem[{Goodman and Colombo(2022)}]{goodman2022collision}
Goodman, J.R. and Colombo, L.J. (2022).
\newblock Collision avoidance of multiagent systems on riemannian manifolds.
\newblock \emph{SIAM Journal on Control and Optimization}, 60(1), 168--188.

\bibitem[{Goodman and Colombo(2023)}]{goodman2022reduction}
Goodman, J.R. and Colombo, L.J. (2023).
\newblock Reduction by symmetry in obstacle avoidance problems on riemannian manifolds.
\newblock \emph{SIAM Journal on Applied Algebra and Geometry}.

\bibitem[{Jost(2008)}]{Jost}
Jost, J. (2008).
\newblock \emph{Riemannian geometry and geometric analysis}.
\newblock Springer.

\bibitem[{Noakes et~al.(1989)Noakes, Heinzinger, and Paden}]{noakes}
Noakes, L., Heinzinger, G., and Paden, B. (1989).
\newblock Cubic splines on curved spaces.
\newblock \emph{IMA Journal of Mathematical Control and Information}, 6(4), 465--473.

\end{thebibliography}
\end{document}